\documentclass[11pt]{article}

\usepackage{amsmath}
\usepackage{graphicx}
\usepackage{yhmath}
\usepackage{mathdots}
\usepackage[nottoc,numbib]{tocbibind}
\usepackage{mathtools}
\usepackage[all,cmtip]{xy}
 \usepackage{amsfonts}
\usepackage{amsthm}
\usepackage{amsmath,amscd}

\usepackage{amssymb}
\usepackage{tikz}
\usetikzlibrary{arrows.meta, decorations.markings}

\usepackage[margin=1in]{geometry}
\usepackage[T1]{fontenc}
\usepackage[utf8]{inputenc}
\usepackage{lmodern}
\usepackage{microtype}
\usepackage{amsmath,amssymb,amsthm,mathtools}
\usepackage{enumitem}
\usepackage{tikz-cd}
\usepackage{array}
\usepackage{xcolor}
\usepackage{hyperref}

\hypersetup{
  colorlinks=true,
  linkcolor=blue!50!black,
  citecolor=blue!50!black,
  urlcolor=blue!50!black
}

\newcommand{\K}{\Bbbk}
\newcommand{\rad}{\mathfrak r}
\newcommand{\In}{\operatorname{In}}
\newcommand{\Out}{\operatorname{Out}}

\theoremstyle{plain}
\newtheorem{theorem}{Theorem}[section]
\newtheorem{proposition}[theorem]{Proposition}

\newtheorem{corollary}[theorem]{Corollary}

\theoremstyle{definition}
\newtheorem{definition}[theorem]{Definition}
\newtheorem{construction}[theorem]{Construction}
\newtheorem{example}[theorem]{Example}
\newtheorem{nonexample}[theorem]{Non-example}
\newtheorem{question}[theorem]{Question}

\theoremstyle{remark}
\newtheorem{remark}[theorem]{Remark}

\title{On the Mutations of  Gentle  Quivers and Admissible Ideals}
\author{Ibrahim Saleh}
\date{\today}

\begin{document}
\maketitle

\begin{abstract}

Fomin–Zelevinsky quiver mutation is an essential tool in the theory of cluster algebras. In this paper, we extend this framework to quivers with relations by introducing an ``involutive'' mutation operation on admissible ideals of path algebras.
 Starting with a bound quiver algebra $(Q,I)$, we extend quiver mutation $\mu_{k}$ so that $(\mu_k(Q),I')$ is again a bound quiver algebra. Throughout this process, we used the \emph{permitted transition graph} associated with each pair $(Q,I)$, which provides a visual test to determine if $(Q,I)$ is a bound quiver algebra. In addition, we give criteria ensuring that $(\mu_k(Q),I')$ is gentle whenever $(Q,I)$ is a gentle algebra, and we identify classes of quivers for which gentleness is preserved under mutation. Finally, we briefly explore some relations among the families of gentle algebras obtained by mutating an initial gentle algebra.

\end{abstract}


\tableofcontents

\section{Introduction}

Mutation is the central concept in the theory of cluster algebra.
Introduced by Fomin and Zelevinsky [3], it is a combinatorial operation for
generating  cluster structures using an initial quiver. 

Gentle algebras, introduced by Assem and Skowroński [1], satisfy some combinatorial criteria on quivers and local quadratic relations, called admissible ideals, (Definitions 3.2, 3.3). Gentle algebras play an important role in the theory of finite-dimensional algebras, and   are closely related to string and special biserial
algebras, and occur naturally in the study of surfaces, derived categories,
and cluster theory.  

In this paper we extend mutation beyond quivers so that it also acts on relations.  Starting with an ``initial'' bound quiver \((Q,I)\), we
construct, a class of new bound quivers by applying a single ``\emph{extended}'' mutation $\mu_{k}$. And we identify which bound quiver produced by applying $\mu_{k}$ on \((Q,I)\) such that $\mu_{k}(\mu_k(Q),I')=(Q,I)$.  For every pair  \((Q,I)\) we associate a \emph{permitted transition graph} which is used to determine the admissibility of the ideal $I$. The permitted transition graph and some local matching conditions are key tools in the construction of the  mutation on relations for producing new admissible ideals. We also generalize this process to  gentle algebras as well, when applicable.

The admissibility  is determined by the permitted
transition graph: by Proposition~\ref{prop:admissibility}, a quadratic
monomial ideal is admissible precisely when its permitted transition graph has
no oriented cycle. This makes it possible to construct the mutated ideal by
choosing permitted transitions first and then defining the zero-relations as
the complement, as in Construction~\ref{constr:admissibilityfirst} and
Proposition~\ref{prop:admissibility-first-guarantee}.  Once the mutated quiver
is \emph{degree-gentle}, the same trick can be refined to the gentle setting:
Proposition~\ref{prop:gentle-admissible-completion} shows that an acyclic
local matching yields a gentle completion, while
Proposition~\ref{prop:decorated-admissibility-reversible} and
Proposition~\ref{prop:decorated-gentle-involutive} show that the resulting
admissibility and gentle mutations become reversible after passing to
decorated bound quivers.  The uniqueness statements in
Proposition~\ref{prop:unique-admissibility-reversible} and
Proposition~\ref{prop:uniqueness-reversible} identify the cases in which the
undecorated constructions are already reversible.

Quivers of finite representations type  provide favorable situations for these constructions.  In the case of type \(A_n\), the
mutation procedure to produce gentle algebras can often be completed successfully, and it gives a large
family of admissible and gentle examples.  In general, these findings suggest that the relevant obstructions are  local, determined by specific subquiver, in particular
 whether the mutated quiver admits an acyclic choice of its permitted transitions graph  and, in the gentle case,
whether this choice satisfies the local matching condition.

Briefly our main motive, is to develop  systematic framework for producing new admissible and gentle bound quiver algebras from
given ones. So, for every cluster structure we associate classes of bound and gentle algebras. Finally we  also suggest several directions for future work, including
the study of mutation classes of gentle algebras, the relation with derived
equivalence, and the possibility of developing a broader decorated mutation
theory for bound quivers with relations.

\section{Preliminaries}

Throughout the paper, \(\K\) denotes a field of zero characteristic.  We begin by recalling the basic
notions of quivers and path algebras that will be used in the rest of the paper.

\begin{definition}[Quivers and paths]
A \emph{quiver} \(Q\) is a finite directed graph consisting of a set of
vertices \(Q_0\), a set of arrows \(Q_1\), and source and target maps
\[
s,t:Q_1\to Q_0.
\]
If \(\alpha\in Q_1\), then \(s(\alpha)\) is the starting vertex of \(\alpha\)
and \(t(\alpha)\) is its ending vertex.

A \emph{path} in \(Q\) is a sequence of arrows
\[
p=\alpha_n\cdots \alpha_2\alpha_1
\]
such that \(t(\alpha_i)=s(\alpha_{i+1})\) for all \(1\le i<n\).  The length of
\(p\) is \(n\).  For each vertex \(i\in Q_0\), we also regard the trivial path
\(e_i\) at \(i\) as a path of length \(0\). A \emph{loop} at vertex \(i\), is a non-trivial arrow from  \(i\) to itself. 

If \(p=\alpha_n\cdots \alpha_1\) and \(q=\beta_m\cdots \beta_1\) are paths with
\(t(p)=s(q)\), then their concatenation is the path
\[
qp=\beta_m\cdots \beta_1\alpha_n\cdots \alpha_1.
\]
If \(t(p)\ne s(q)\), then the product \(qp\) is not defined.
\end{definition}

\begin{definition}[Path algebras]
 
The \emph{path algebra} \(\K Q\) is the \(\K\)-vector space having as basis all
paths in \(Q\), with multiplication induced by concatenation of paths when the
product is defined and \(0\) otherwise.  In particular,
\[
e_i e_i=e_i, \ \    \ \ e_ie_j=0, \  \text{if} \  \  i\ne j,
\qquad \text{and} \ \ 
e_{t(\alpha)}\alpha=\alpha=\alpha e_{s(\alpha)}
\]
for every arrow \(\alpha\in Q_1\).

\end{definition}

\begin{definition}
Let $Q$ be a finite connected quiver. The \emph{arrow ideal} of \(\K Q\), denoted \(\rad_Q\), is the ideal generated by
all arrows of \(Q\).  Equivalently, \(\rad_Q\) is the \(\K\)-subspace spanned by
all paths of positive length.  
 As a vector space, the arrow ideal can be decomposed as
\[
\rad_{Q} = \K Q_{1} \oplus \K Q_{2} \oplus \cdots \oplus \K Q_{l} \oplus \cdots ,
\]
where \(\K Q_{l}\) is the subspace of \(\K Q\) with the set $Q_{l}$ of paths of length $l$ as basis.  

The $l$-th power of the arrow ideal can be written as 
\[
\rad_{Q}^{\,l} =\text{the ideal spanned by all paths of length at least }l= \bigoplus_{m \geq l} \K Q_{m}.
\]

An ideal \(I\subseteq \K Q\) is \emph{admissible} if there exists an integer
\(m\geq 2\) such that
\[
  \rad_Q^m\subseteq I\subseteq \rad_Q^2.
\]
In this case, the pair \((Q,I)\) is called a \emph{bound quiver}, and the quotient algebra \(\K Q/I\) is the
corresponding \emph{bound quiver algebra}.
\end{definition}

\begin{definition}
A bound quiver algebra \(\K Q/I\) is \emph{gentle} if the following conditions
hold.
\begin{enumerate}[label=(G\arabic*),leftmargin=2em]
  \item At each vertex of \(Q\), there are at most two incoming arrows and at
  most two outgoing arrows.
  \item For every arrow \(\alpha\in Q_1\), there is at most one arrow \(\beta\)
  such that \(\beta\alpha\notin I\), and at most one arrow \(\gamma\) such that
  \(\alpha\gamma\notin I\), whenever the compositions are defined.
  \item The ideal \(I\) is generated by paths of length two.
  \item For every arrow \(\alpha\in Q_1\), there is at most one arrow \(\beta\)
  such that \(\beta\alpha\in I\), and at most one arrow \(\gamma\) such that
  \(\alpha\gamma\in I\), whenever the compositions are defined.
\end{enumerate}
\end{definition}

\begin{definition} [Cluster quivers] A quiver \(Q\) is called a \emph{cluster quiver} if it contains neither loops nor $2$-cycles.

\end{definition}

In the rest of the article all quivers are finite cluster quivers unless otherwise stated. 

\begin{definition}[Quivers mutation]
We recall the quiver mutation rules, [8].  Let \(Q\) be such a quiver and let \(k\in Q_0\).  The mutation
\(\mu_k(Q)\) is obtained by the following steps.
\begin{enumerate}[leftmargin=2em]
  \item For every oriented path \(i\to k\to j\), add an arrow \(i\to j\).
  \item Reverse all arrows incident with \(k\).
  \item Cancel all oriented \(2\)-cycles.
\end{enumerate}
The resulting quiver is denoted by \(Q'=\mu_k(Q)\). 
\end{definition}
\section{Mutation of bound quivers and gentle algebras}

We start with notions that will be used  to define mutations admissible ideals. 
\[
  I_k=\{i\in Q_0\mid i\to k\text{ is an arrow of }Q\},
  \qquad  \text{and} \qquad 
  O_k=\{j\in Q_0\mid k\to j\text{ is an arrow of }Q\}.
\]
The vertices in \(I_k\) are the incoming neighbors of \(k\), and the vertices in
\(O_k\) are the outgoing neighbors of \(k\).

\subsection{Exploring the invariance of gentle-degree quivers under mutation}

\begin{definition}
A quiver \(Q\) is called \emph{degree-gentle} if it satisfies only the arrow-degree condition, (G1),  in the definition of a gentle algebra above.  Equivalently,
\[
  \deg_Q^-(v)\leq 2
  \qquad\text{and}\qquad
  \deg_Q^+(v)\leq 2
\]
for every vertex \(v\in Q_0\).
\end{definition}

Thus every gentle algebra has a degree-gentle underlying quiver, but a
degree-gentle quiver does not determine a gentle algebra until relations are
chosen. In rest of this section we show that the degree-gentle condition after mutation can be
checked purely locally.

\begin{proposition}[Local degree criterion]\label{prop:degree-criterion}
Let \(Q\) be a finite simple \(2\)-acyclic degree-gentle quiver, and let
\(k\in Q_0\).  For \(i\in I_k\), define
\[
  c_i=\#\{j\in O_k\mid j\to i\text{ is an arrow of }Q\},
\]
and for \(j\in O_k\), define
\[
  c^j=\#\{i\in I_k\mid j\to i\text{ is an arrow of }Q\}.
\]
Then \(\mu_k(Q)\) is degree-gentle if and only if the following inequalities
hold:
\[
  \deg_Q^+(i)-1+|O_k|-c_i\leq 2,
  \qquad
  \deg_Q^-(i)+1-c_i\leq 2
\]
for every \(i\in I_k\), and
\[
  \deg_Q^-(j)-1+|I_k|-c^j\leq 2,
  \qquad
  \deg_Q^+(j)+1-c^j\leq 2
\]
for every \(j\in O_k\).
\end{proposition}

\begin{proof}
Let \(Q'=\mu_k(Q)\).  Since \(Q\) is \(2\)-acyclic, the sets \(I_k\) and \(O_k\)
are disjoint.  Mutation at \(k\) changes only the arrows incident with \(k\) and
the arrows between vertices of \(I_k\) and vertices of \(O_k\).  Therefore every
vertex outside
\[
  I_k\cup O_k\cup\{k\}
\]
has the same incoming and outgoing degree in \(Q'\) as it had in \(Q\).

At the vertex \(k\), all incident arrows are reversed.  Hence
\[
  \deg_{Q'}^+(k)=\deg_Q^-(k),
  \qquad
  \deg_{Q'}^-(k)=\deg_Q^+(k).
\]
Since \(Q\) is degree-gentle, both numbers are at most two.  Thus \(k\) itself
causes no obstruction.

Let \(i\in I_k\).  The arrow \(i\to k\) is reversed, so \(i\) loses one outgoing
arrow and gains one incoming arrow.  For each \(j\in O_k\), the path
\(i\to k\to j\) creates a new arrow \(i\to j\).  If \(j\to i\) was already an
arrow of \(Q\), then the new arrow \(i\to j\) cancels with \(j\to i\).  The
number of such cancellations involving \(i\) is \(c_i\).  Therefore
\[
  \deg_{Q'}^+(i)=\deg_Q^+(i)-1+|O_k|-c_i,
\]
and
\[
  \deg_{Q'}^-(i)=\deg_Q^-(i)+1-c_i.
\]

Similarly, let \(j\in O_k\).  The arrow \(k\to j\) is reversed, so \(j\) loses
one incoming arrow and gains one outgoing arrow.  For each \(i\in I_k\), the
path \(i\to k\to j\) creates a new arrow \(i\to j\), unless it cancels with an
already existing arrow \(j\to i\).  The number of such cancellations involving
\(j\) is \(c^j\).  Hence
\[
  \deg_{Q'}^-(j)=\deg_Q^-(j)-1+|I_k|-c^j,
\]
and
\[
  \deg_{Q'}^+(j)=\deg_Q^+(j)+1-c^j.
\]

We have computed all degree changes.  Since all unaffected vertices remain
degree-gentle and \(k\) automatically remains degree-gentle, the mutated quiver
\(Q'\) is degree-gentle exactly when the displayed inequalities hold at the
vertices in \(I_k\cup O_k\).
\end{proof}

\begin{corollary}[Source and sink mutations]\label{cor:source-sink}
Let \(Q\) be a finite simple \(2\)-acyclic degree-gentle quiver.
\begin{enumerate}[label=(\arabic*),leftmargin=2em]
  \item If \(k\) is a source, then \(\mu_k(Q)\) is degree-gentle if and only if
  \[
    \deg_Q^+(j)\leq 1
  \]
  for every arrow \(k\to j\).
  \item If \(k\) is a sink, then \(\mu_k(Q)\) is degree-gentle if and only if
  \[
    \deg_Q^-(i)\leq 1
  \]
  for every arrow \(i\to k\).
\end{enumerate}
\end{corollary}

\begin{proof}
If \(k\) is a source, then \(I_k=\varnothing\), so mutation creates no new
arrows through \(k\).  It only reverses the arrows \(k\to j\).  Each such
neighbor \(j\) loses one incoming arrow and gains one outgoing arrow, so
\[
  \deg_{\mu_k(Q)}^-(j)=\deg_Q^-(j)-1,
  \qquad
  \deg_{\mu_k(Q)}^+(j)=\deg_Q^+(j)+1.
\]
The incoming degree is automatically at most two, while the outgoing degree is
at most two exactly when \(\deg_Q^+(j)\leq 1\).  The sink case is dual.
\end{proof}

\begin{remark}
The condition that \(k\) is a source or a sink is not sufficient by itself.  A
neighbor of \(k\) can become over-saturated after the incident arrow is reversed.
The extra inequalities in Corollary \ref{cor:source-sink} are necessary.
\end{remark}

\subsection{Matchings relations and gentleness property}

We now turn from the quiver to the ideal of relations.  In this section, \(Q\)
is a degree-gentle quiver and \(I\subseteq \K Q\) is a monomial quadratic ideal.

\begin{definition}
The set of zero-relations of \(I\) is
\[
  Z(I)=\{(\alpha,\beta)\mid \alpha:i\to j,
  \ \beta:j\to \ell,\ \beta\alpha\in I\}.
\]
Thus \(Z(I)\) records exactly which length-two paths are killed.
\end{definition}

Let \(v\in Q_0\).  Write
\[
  \In_Q(v)=\{\alpha:u\to v\},
  \qquad
  \Out_Q(v)=\{\beta:v\to w\}.
\]
The length-two paths passing through \(v\) are the elements of
\[
  \In_Q(v)\times \Out_Q(v).
\]
For a chosen set of zero-relations \(Z\), define the \emph{local part} of \(v\) by
\[
  Z_v=Z\cap \bigl(\In_Q(v)\times \Out_Q(v)\bigr).
\]

\begin{definition}
A subset \(S\subseteq X\times Y\) is a \emph{matching} if no two elements of
\(S\) have the same first coordinate and no two elements of \(S\) have the same
second coordinate.
\end{definition}

\begin{proposition}[Local matching criterion]\label{prop:matching}
Let \(Q\) be degree-gentle and let \(I=\langle Z\rangle\) be generated by
length-two paths.  Then the local continuation conditions of gentleness hold for
\(\K Q/I\) if and only if, for every vertex \(v\in Q_0\), both
\[
  Z_v
  \qquad\text{and}\qquad
  \bigl(\In_Q(v)\times \Out_Q(v)\bigr)\setminus Z_v
\]
are matchings.
\end{proposition}

\begin{proof}
The set \(Z_v\) records the forbidden length-two paths through \(v\).  It is a
matching precisely when each incoming arrow has at most one forbidden right
continuation and each outgoing arrow has at most one forbidden left
continuation.  This is exactly condition (G4) at \(v\).

The complement records the permitted length-two paths through \(v\).  It is a
matching precisely when each incoming arrow has at most one permitted right
continuation and each outgoing arrow has at most one permitted left
continuation.  This is exactly condition (G2) at \(v\).  Since \(Q\) is already
degree-gentle and \(I\) is generated by length-two paths, these matching
conditions are equivalent to the local gentle conditions.
\end{proof}

Because \(|\In_Q(v)|\leq 2\) and \(|\Out_Q(v)|\leq 2\), Proposition
\ref{prop:matching} gives the following table.
\[
\begin{array}{c|l}
(|\In_Q(v)|,|\Out_Q(v)|) & \text{Allowed choices for }Z_v \\
\hline
(0,b)\text{ or }(a,0) & \varnothing \\
(1,1) & \text{kill the unique path, or allow it} \\
(1,2) & \text{kill exactly one of the two paths} \\
(2,1) & \text{kill exactly one of the two paths} \\
(2,2) & \text{kill a diagonal or anti-diagonal pair of paths.}
\end{array}
\]
In particular, if
\[
  \In_Q(v)=\{\alpha_1,\alpha_2\},
  \qquad
  \Out_Q(v)=\{\beta_1,\beta_2\},
\]
then in the \((2,2)\)-case the allowed choices are
\[
  Z_v=\{(\alpha_1,\beta_1),(\alpha_2,\beta_2)\}
\]
or
\[
  Z_v=\{(\alpha_1,\beta_2),(\alpha_2,\beta_1)\}.
\]

\subsection{Admissibility and the permitted transition graph}

The matching criterion is local, but admissibility is global. We now introduce a graph associated to each pair $(Q,I)$, a quiver $Q$ and monomial quadratic ideal \(I\)  which we use to identify a simple global admissibility test.

\begin{definition}
Let \(Q\) be a finite quiver and let \(Z\) be a set of zero length-two paths.
The \emph{permitted transition graph} \(T(Q,Z)\) is the directed graph defined
as follows:
\begin{itemize}[leftmargin=2em]
  \item the vertices of \(T(Q,Z)\) are the arrows of \(Q\);
  \item there is an arrow \(\alpha\to \beta\) in \(T(Q,Z)\) if
  \(\beta\alpha\) is a composable length-two path in \(Q\) and
  \((\alpha,\beta)\notin Z\).
\end{itemize}
\end{definition}

\begin{proposition}[Admissibility test]\label{prop:admissibility}
Let \(Q\) be a finite quiver and let \(I=\langle Z\rangle\) be a monomial
quadratic ideal.  Then \(I\) is admissible if and only if \(T(Q,Z)\) has no
oriented cycle.
\end{proposition}

\begin{proof}
If \(T(Q,Z)\) has an oriented cycle, then one can follow this cycle repeatedly.
This produces arbitrarily long paths in \(Q\) whose consecutive length-two
subpaths are all permitted.  Hence no power \(\rad_Q^m\) can be contained in
\(I\), and \(I\) is not admissible.

Conversely, suppose \(T(Q,Z)\) has no oriented cycle.  Since \(Q\) is finite,
\(T(Q,Z)\) is finite.  Therefore there is a maximum length of a directed path in
\(T(Q,Z)\).  Equivalently, there is a maximum length of a path in \(Q\) that
survives modulo \(I\).  Hence all sufficiently long paths in \(Q\) lie in \(I\),
so \(\rad_Q^m\subseteq I\) for some \(m\geq 2\).  Since \(I\) is generated by
length-two paths, \(I\subseteq \rad_Q^2\).  Thus \(I\) is admissible.
\end{proof}

\begin{corollary}\label{cor:gentlecriterion}
Let \(Q\) be degree-gentle and let \(I=\langle Z\rangle\) be a monomial
quadratic ideal.  Then \(\K Q/I\) is gentle if and only if
\begin{enumerate}[label=(\roman*),leftmargin=2em]
  \item at every vertex \(v\), both \(Z_v\) and its complement in
  \(\In_Q(v)\times \Out_Q(v)\) are matchings; and
  \item the permitted transition graph \(T(Q,Z)\) is acyclic.
\end{enumerate}
\end{corollary}

\begin{example}[A non-admissible ideal and a cyclic permitted transition graph]
Let \(Q\) be the oriented \(3\)-cycle

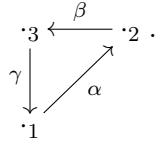
\begin{equation}\label{}
\nonumber \xymatrix{
\cdot_{3} \ar[d]_{\gamma} & \ar[l]_{\beta}\cdot_{2} \\
	\cdot_{1}  \ar[ur]_{\alpha}}.
  \end{equation}

Take
\[
Z=\varnothing,
\qquad
I=\langle Z\rangle=0.
\]
Then every length-two path is permitted. Hence the permitted transition graph
\(T(Q,Z)\) has vertices
\[
\alpha,\ \beta,\ \gamma,
\]
and arrows

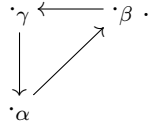
\begin{equation}\label{}
\nonumber \xymatrix{
\cdot_{\gamma} \ar[d] & \ar[l]\cdot_{\beta} \\
	\cdot_{\alpha}  \ar[ur]}.
  \end{equation}

Thus \(T(Q,Z)\) is itself an oriented \(3\)-cycle. Therefore \(T(Q,Z)\) has an oriented cycle.

On the algebra side, the ideal \(I=0\) is not admissible. Indeed, for every
\(m\geq 2\), there are nonzero paths of length \(m\) obtained by going around
the cycle repeatedly. Hence
\[
\rad_Q^m \nsubseteq I=0
\]
for every \(m\geq 2\). Therefore there is no integer \(m\geq 2\) such that
\[
\rad_Q^m\subseteq I\subseteq \rad_Q^2.
\]
Thus \(I\) is not admissible, in agreement with Proposition~\ref{prop:admissibility}.
\end{example}

\begin{example}[An admissible ideal obtained by breaking the permitted cycle]
Let \(Q\) again be the oriented \(3\)-cycle

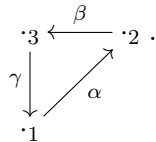
\begin{equation}\label{}
\nonumber \xymatrix{
\cdot_{3} \ar[d]_{\gamma} & \ar[l]_{\beta}\cdot_{2} \\
	\cdot_{1}  \ar[ur]_{\alpha}}.
  \end{equation}
  
This time take
\[
Z=\{(\alpha,\beta)\},
\qquad
I=\langle \beta\alpha\rangle.
\]
Thus the length-two path
\[
1\xrightarrow{\alpha}2\xrightarrow{\beta}3
\]
is zero, while the other two length-two paths are permitted:
\[
\gamma\beta\notin I,
\qquad
\alpha\gamma\notin I.
\]

The permitted transition graph \(T(Q,Z)\) has vertices
\[
\alpha,\ \beta,\ \gamma.
\]
Since \((\alpha,\beta)\in Z\), there is no arrow
\[
\alpha\longrightarrow \beta
\]
in \(T(Q,Z)\). However, the other two transitions are permitted, so we have
\[
\beta\longrightarrow \gamma,
\qquad
\gamma\longrightarrow \alpha.
\]
Therefore
\[
T(Q,Z):
\qquad
\beta\longrightarrow \gamma\longrightarrow \alpha
\]
has no oriented cycle.

We now check admissibility directly. Clearly,
\[
I=\langle \beta\alpha\rangle\subseteq \rad_Q^2.
\]
Moreover, every path of length \(4\) in the oriented \(3\)-cycle contains the
subpath \(\beta\alpha\). Indeed, any path of length \(4\) winds around the
cycle far enough to include the consecutive pair
\[
\alpha,\beta,
\]
and hence contains the zero relation \(\beta\alpha\). Therefore
\[
\rad_Q^4\subseteq I.
\]
Thus
\[
\rad_Q^4\subseteq I\subseteq \rad_Q^2,
\]
so \(I\) is admissible. This agrees with Proposition~\ref{prop:admissibility},
because \(T(Q,Z)\) is acyclic.
\end{example}

\begin{example}[A nonzero non-admissible ideal]
Let \(Q\) be the quiver

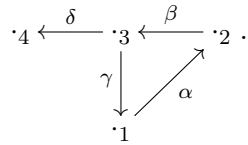
\begin{equation}\label{}
\nonumber \xymatrix{
\cdot_{4}&\cdot_{3} \ar[d]_{\gamma} \ar[l]_{\delta}& \ar[l]_{\beta}\cdot_{2} \\
	&\cdot_{1}  \ar[ur]_{\alpha}}.
  \end{equation}

Let
\[
Z=\{(\beta,\delta)\},
\qquad
I=\langle \delta\beta\rangle.
\]
Thus the path
\[
2\xrightarrow{\beta}3\xrightarrow{\delta}4
\]
is zero, but the transitions around the oriented \(3\)-cycle remain permitted, therefore \(T(Q,Z)\) contains the oriented cycle

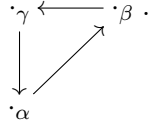
\begin{equation}\label{}
\nonumber \xymatrix{
\cdot_{\gamma} \ar[d] & \ar[l]\cdot_{\beta} \\
	\cdot_{\alpha}  \ar[ur]}.
  \end{equation}

Consequently, \(I\) is not admissible. Indeed, for every \(m\geq 2\), there
are nonzero paths of length \(m\) obtained by repeatedly traveling around the
cycle

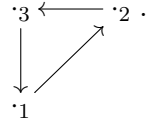
\begin{equation}\label{}
\nonumber \xymatrix{
\cdot_{3} \ar[d] & \ar[l]\cdot_{2} \\
	\cdot_{1}  \ar[ur]}.
  \end{equation}

These paths never use the zero relation \(\delta\beta\). Hence
\[
\rad_Q^m\nsubseteq I
\]
for every \(m\geq 2\), and so no integer \(m\) satisfies
\[
\rad_Q^m\subseteq I\subseteq \rad_Q^2.
\]
Thus \(I\) is not admissible, again in agreement with
Proposition~\ref{prop:admissibility}.
\end{example}

\section{A mutation extension for relations of \((Q, I)\).}

Let \((Q,I)\) be the initial gentle bound quiver, let \(k\in Q_0\), and suppose that
\[
  Q'=\mu_k(Q)
\]
is degree-gentle.  We describe a procedure for choosing a candidate ideal
\(I'\subseteq \K Q'\). A mutation-based operation to be applied on a bound quiver algebra \((Q, I)\) to produce a new one.

Write
\[
  I_k=\{i\in Q_0\mid i\to k\},
  \qquad
  O_k=\{j\in Q_0\mid k\to j\}.
\]
For \(i\in I_k\), write
\[
  \alpha_i:i\to k,
\]
and for \(j\in O_k\), write
\[
  \beta_j:k\to j.
\]
After mutation, these arrows become
\[
  \alpha_i^*:k\to i,
  \qquad
  \beta_j^*:j\to k.
\]
For every initial path
\[
  i\xrightarrow{\alpha_i} k \xrightarrow{\beta_j} j,
\]
mutation creates a composite arrow
\[
  \eta_{ij}:i\to j,
\]
unless this arrow is canceled with an opposite arrow during the \(2\)-cycle
cancellation step.

\subsection{Mutation of Admissible Ideals}
\label{sec:admissibility-first-relation-mutation}

In this subsection we introduce  Construction ~\ref{constr:admissibilityfirst}. Our purpose is extend the action of mutation to also cover the permissible ideals. So we start with a pair of a quiver and an admissible ideal \((Q,I)\) and produce another pair \((\mu_{k}(Q),I')\) where \(I'\) is an admissible ideal of \(K(\mu_{k}(Q))\).

  The key idea is to use the admissibility test, Proposition ~\ref{prop:admissibility},  by constructing the permitted transitions
first and to choose them to be acyclic.  The zero-relations are then defined
as the complement of the chosen permitted transitions.  Since admissibility of
quadratic monomial ideals is equivalent to acyclicity of the permitted
transition graph, this procedure produces an admissible ideal by construction.

Throughout this section, let \((Q,I)\) be a gentle bound quiver, let
\[
  Q'=\mu_k(Q),
\]
and assume that
\[
  I=\langle Z\rangle
\]
is a quadratic monomial ideal.  Let
\[
  \mathcal C(Q')
  =
  \{(\alpha,\beta)\in Q'_1\times Q'_1
    \mid t(\alpha)=s(\beta)\}
\]
denote the set of composable pairs of arrows in \(Q'\).  Thus a pair
\((\alpha,\beta)\in \mathcal C(Q')\) corresponds to the length-two path
\(\beta\alpha\).

\begin{construction}[Admissibility mutation]
\label{constr:admissibilityfirst}
We construct a quadratic monomial ideal
\[
  I'=\langle Z'\rangle\subseteq \K Q'
\]
as follows.

\begin{enumerate}[label=\textbf{Step \arabic*.},leftmargin=3em]
  \item \textbf{Record the mutation data.}
  For every oriented path
  \[
    i\xrightarrow{\alpha_i} k\xrightarrow{\beta_j} j
  \]
  in \(Q\), let
  \[
    \eta_{ij}:i\to j
  \]
  denote the corresponding composite arrow in \(Q'\), provided that this arrow
  survives the two-cycle cancellation step.  Also write
  \[
    \alpha_i^*:k\to i,
    \qquad
    \beta_j^*:j\to k
  \]
  for the arrows obtained by reversing \(\alpha_i\) and \(\beta_j\), respectively.

  \item \textbf{Define the forced zero-relations.}
  Let
  \[
    Z_{\mathrm{for}}\subseteq \mathcal C(Q')
  \]
  be the set of length-two paths that are forced to be zero after mutation.
  We define \(Z_{\mathrm{for}}\) by the following rules.

  \begin{enumerate}[label=(\alph*),leftmargin=2.5em]
    \item If a length-two path in \(Q'\) is made of initial arrows that survive the
    mutation and is not affected by the mutation at \(k\), then preserve its
    initial zero-relation status.  Thus, if
    \[
      \beta\alpha\in I,
    \]
    then put
    \[
      (\alpha,\beta)\in Z_{\mathrm{for}}.
    \]

    \item If an initial arrow is removed during the two-cycle cancellation step,
    then delete every initial zero-relation involving that arrow.  Such a relation
    cannot be transported directly to \(Q'\).

    \item Suppose that \(\eta_{ij}:i\to j\) is a new composite arrow in \(Q'\).
    If \(\rho:h\to i\) is an initial arrow that survives in \(Q'\), then put
    \[
      (\rho,\eta_{ij})\in Z_{\mathrm{for}}
      \quad\text{whenever}\quad
      \alpha_i\rho\in I.
    \]
    Equivalently,
    \[
      \eta_{ij}\rho\in I'
      \quad\text{whenever}\quad
      \alpha_i\rho\in I.
    \]

    Similarly, if \(\tau:j\to \ell\) is an initial arrow that survives in \(Q'\),
    then put
    \[
      (\eta_{ij},\tau)\in Z_{\mathrm{for}}
      \quad\text{whenever}\quad
      \tau\beta_j\in I.
    \]
    Equivalently,
    \[
      \tau\eta_{ij}\in I'
      \quad\text{whenever}\quad
      \tau\beta_j\in I.
    \]

    \item Through the mutated vertex \(k\), use the complement-transpose rule.
    Thus, for every initial path
    \[
      i\xrightarrow{\alpha_i}k\xrightarrow{\beta_j}j,
    \]
    put
    \[
      (\beta_j^*,\alpha_i^*)\in Z_{\mathrm{for}}
      \quad\text{whenever}\quad
      \beta_j\alpha_i\notin I.
    \]
    Equivalently,
    \[
      \alpha_i^*\beta_j^*\in I'
      \quad\text{whenever}\quad
      \beta_j\alpha_i\notin I.
    \]
    Hence an initial permitted path through \(k\) becomes a zero path through
    \(k\) after mutation, while an initial zero path through \(k\) becomes
    preferred permitted.
  \end{enumerate}

  \item \textbf{Define preferred permitted transitions.}
  Let
  \[
    P_{\mathrm{pref}}\subseteq \mathcal C(Q')\setminus Z_{\mathrm{for}}
  \]
  be the set of composable pairs that we would like to keep permitted.  This
  set contains, for example,
  \begin{enumerate}[label=(\alph*),leftmargin=2.5em]
    \item initial permitted length-two paths away from the mutation neighborhood;
    \item transported permitted paths involving new composite arrows;
    \item reversed paths through \(k\) coming from initial zero-relations through
    \(k\).
  \end{enumerate}
  We impose the convention that zero-relations dominate: if a pair belongs to
  \(Z_{\mathrm{for}}\), then it is not allowed to belong to \(P_{\mathrm{pref}}\).

  \item \textbf{Choose an acyclic permitted skeleton.}
  Choose a subset
  \[
    P\subseteq P_{\mathrm{pref}}
  \]
  such that the directed graph with vertex set \(Q'_1\) and arrows
  \[
    \alpha\longrightarrow \beta
    \qquad\text{for }(\alpha,\beta)\in P
  \]
  has no oriented cycle.  There are two systematic ways to make such a choice.

  \begin{enumerate}[label=(\alph*),leftmargin=2.5em]
    \item If the graph determined by \(P_{\mathrm{pref}}\) is already acyclic,
    set
    \[
      P=P_{\mathrm{pref}}.
    \]

    \item Otherwise, choose a total order \(\prec\) on the set of arrows
    \(Q'_1\), and set
    \[
      P=
      \{(\alpha,\beta)\in P_{\mathrm{pref}}
        \mid \alpha\prec \beta\}.
    \]
    Then every permitted transition strictly increases with respect to
    \(\prec\), and therefore \(P\) is automatically acyclic. Equivalently, one may choose a feedback edge set \(B\) meeting every
    oriented cycle in the directed graph determined by \(P_{\mathrm{pref}}\),
    and set
    \[
      P=P_{\mathrm{pref}}\setminus B.
    \]
  \end{enumerate}

  \item \textbf{Define the new zero-relation set.}
  Set
  \[
    Z'=\mathcal C(Q')\setminus P.
  \]
  Equivalently, a length-two path \(\beta\alpha\) is zero in \(I'\) if and only
  if
  \[
    (\alpha,\beta)\notin P.
  \]
  Finally, define
  \[
    I'
    =
    \langle \beta\alpha\mid (\alpha,\beta)\in Z'\rangle
    \subseteq \K Q'.
  \]
\end{enumerate}
\end{construction}

\begin{proposition}[Admissibility guarantee]
\label{prop:admissibility-first-guarantee}
Let \(Q'=\mu_k(Q)\), and let
\[
  I'=\langle Z'\rangle\subseteq \K Q'
\]
be constructed as in Construction~\ref{constr:admissibilityfirst}.  Then
\(I'\) is admissible.
\end{proposition}

\begin{proof}
By construction, \(I'\) is generated by length-two paths. Hence
\[
  I'\subseteq \rad_{Q'}^2.
\]
It remains to show that \(\rad_{Q'}^m\subseteq I'\) for some \(m\geq 2\).

The permitted transition graph \(T(Q',Z')\) has vertex set \(Q'_1\), and its
arrows are precisely the transitions
\[
  \alpha\longrightarrow \beta
  \qquad\text{with}\qquad
  (\alpha,\beta)\in P.
\]
By Step 4 of Construction~\ref{constr:admissibilityfirst}, this directed graph
is acyclic.

Since \(Q'_1\) is finite, an acyclic directed graph on \(Q'_1\) has no directed
path of length greater than or equal to \(|Q'_1|\). Therefore every sufficiently
long path in \(Q'\) contains a consecutive length-two subpath which is not
permitted, hence belongs to \(Z'\). Thus, for example,
\[
  \rad_{Q'}^{|Q'_1|+1}\subseteq I'.
\]
Consequently,
\[
  \rad_{Q'}^{|Q'_1|+1}\subseteq I'\subseteq \rad_{Q'}^2,
\]
and so \(I'\) is admissible.
\end{proof}

The construction above guarantees admissibility, but not necessarily gentleness.
To obtain a gentle algebra, the permitted transitions must also satisfy the
local matching condition.  The following definition isolates the extra
condition needed for gentleness.

\begin{definition}[Acyclic local matching]
\label{def:acyclic-local-matching}
Assume that \(Q'\) is degree-gentle.  A subset
\[
  P\subseteq \mathcal C(Q')
\]
is called an \emph{acyclic local matching} if the following conditions hold.

\begin{enumerate}[label=(\alph*),leftmargin=2.5em]
  \item The directed graph with vertex set \(Q'_1\) and arrows
  \[
    \alpha\to \beta
    \qquad\text{for }(\alpha,
    \beta)\in P
  \]
  has no oriented cycle.

  \item For each vertex \(v\in Q'_0\), the set
  \[
    P_v
    =
    P\cap
    \bigl(\operatorname{In}_{Q'}(v)\times
          \operatorname{Out}_{Q'}(v)\bigr)
  \]
  satisfies the following local matching rules:
  \[
  \begin{array}{c|c}
  \bigl(\#\operatorname{In}_{Q'}(v),
        \#\operatorname{Out}_{Q'}(v)\bigr)
  & \text{condition on }P_v \\ \hline
  (0,b)\text{ or }(a,0) & P_v=\varnothing \\
  (1,1) & P_v \text{ has size }0\text{ or }1 \\
  (1,2) & P_v \text{ has size }1 \\
  (2,1) & P_v \text{ has size }1 \\
  (2,2) & P_v \text{ is a perfect matching.}
  \end{array}
  \]
\end{enumerate}
\end{definition}

\begin{construction}[Gentle-admissible completion]
\label{constr:gentle-admissible-completion}
Keep Steps 1--3 of Construction~\ref{constr:admissibilityfirst}.  Instead of
choosing an arbitrary acyclic subset \(P\subseteq P_{\mathrm{pref}}\), choose an
acyclic local matching
\[
  P\subseteq \mathcal C(Q')\setminus Z_{\mathrm{for}}
\]
which agrees with \(P_{\mathrm{pref}}\) as much as possible.  Then set
\[
  Z'=\mathcal C(Q')\setminus P
\]
and
\[
  I'
  =
  \langle \beta\alpha\mid (\alpha,\beta)\in Z'\rangle
  \subseteq \K Q'.
\]
\end{construction}

\begin{proposition}[Gentle-admissible completion]
\label{prop:gentle-admissible-completion}
Assume that \(Q'=\mu_k(Q)\) is degree-gentle.  If
Construction~\ref{constr:gentle-admissible-completion} can be completed, then
\[
  \K Q'/I'
\]
is a gentle algebra.
\end{proposition}

\begin{proof}
Since \(Q'\) is degree-gentle, every vertex has at most two incoming and at
most two outgoing arrows.  The ideal \(I'\) is generated by length-two paths by
construction.

The local matching condition guarantees the permitted and forbidden
continuation conditions.  Indeed, for each arrow \(\alpha\), there is at most
one arrow \(\beta\) such that \((\alpha,\beta)\in P\), and hence at most one
right continuation \(\beta\alpha\notin I'\).  Similarly, there is at most one
left continuation not in \(I'\).  Since
\[
  Z'=\mathcal C(Q')\setminus P,
\]
the same local matching rules imply that there is at most one zero
continuation on each side.

Finally, \(P\) is acyclic.  Therefore \(T(Q',Z')\) has no oriented cycle.  By
Proposition~\ref{prop:admissibility}, the ideal \(I'\) is admissible.  Hence
\(\K Q'/I'\) is gentle.
\end{proof}

\begin{remark}
Construction~\ref{constr:admissibilityfirst} always produces an admissible
quadratic monomial ideal, because acyclicity is built into the choice of the
permitted transition set \(P\).  
However, if one also wants \(\K Q'/I'\) to be gentle, then the permitted
transition set must satisfy the local matching condition.  This is an
additional constraint.  Therefore the gentle version is not automatic: it
exists precisely when there is an acyclic local matching compatible with the
forced zero-relations \(Z_{\mathrm{for}}\).  If no such acyclic local matching
exists, then the obstruction is not merely an artifact of the construction; it
means that the chosen forced relations cannot be completed to a gentle
admissible ideal using this local mutation rule.
\end{remark}

\begin{example}[The admissibility-first construction on an oriented path]
Let

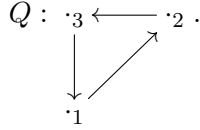
\begin{equation}\label{}
\nonumber  Q: \xymatrix{
\cdot_{3} \ar[d] & \ar[l]\cdot_{2} \\
	\cdot_{1}  \ar[ur]}.
  \end{equation}
  
and let
\[
  I=\langle \beta\alpha\rangle.
\]
Mutate at \(k=2\).  Then \(Q'=\mu_2(Q)\) has arrows
\[
  \alpha^*:2\to 1,
  \qquad
  \eta:1\to 3,
  \qquad
  \beta^*:3\to 2,
\]
and hence \(Q'\) is the oriented three-cycle

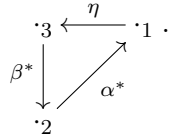
\begin{equation}\label{}
\nonumber   \xymatrix{
\cdot_{3} \ar[d]_{\beta^*} & \ar[l]_{\eta}\cdot_{1} \\
	\cdot_{2}  \ar[ur]_{\alpha^*}}.
  \end{equation}

The initial path \(\beta\alpha\) was zero, so the complement-transpose rule prefers
that the reversed path

\[
  3\xrightarrow{\beta^*}2\xrightarrow{\alpha^*}1
\]
remain permitted.  Thus \((\beta^*,\alpha^*)\) belongs to
\(P_{\mathrm{pref}}\).  If we also prefer to keep the other two length-two paths
permitted, then \(P_{\mathrm{pref}}\) contains
\[
  (\alpha^*,\eta),
  \qquad
  (\eta,\beta^*),
  \qquad
  (\beta^*,\alpha^*).
\]
These three preferred transitions form an oriented cycle in the permitted
transition graph.

Choose the total order
\[
  \eta\prec \beta^*\prec \alpha^*.
\]
The admissibility-first rule keeps only the increasing transitions:
\[
  P=\{(\eta,\beta^*), (\beta^*,\alpha^*)\}.
\]
Therefore
\[
  Z'=\mathcal C(Q')\setminus P
\]
and, among the length-two paths on the three-cycle, this gives
\[
  (\alpha^*,\eta)\in Z'.
\]
Hence one possible admissible ideal is
\[
  I'=\langle \eta\alpha^*\rangle.
\]
The permitted transition graph is now the directed path
\[
  \eta\longrightarrow \beta^*\longrightarrow \alpha^*,
\]
together with isolated vertices.  In particular, it has no oriented cycle.
Therefore \(I'\) is admissible.
\end{example}

\begin{remark}
The guiding principle of this section can be summarized as follows:
\[
  \boxed{
  \text{choose permitted transitions first, make them acyclic, and then define }
  Z'\text{ as their complement.}
  }
\]
This guarantees admissibility.  The gentle version requires the stronger
condition that the chosen acyclic permitted set is also a local matching.
\end{remark}

\subsection{A reversible admissibility criterion}\label{subsec:reversible-admissibility}

Construction~\ref{constr:admissibilityfirst} produces admissible ideals by
choosing a permitted skeleton that is acyclic and then defining the zero
relations as its complement.  Since the choice of the permitted skeleton may
not be unique, the construction is not reversible in general on undecorated
bound quivers.  To recover reversibility, we record the local admissibility
choices in a certificate.

\begin{definition}[The set of admissible completions]
Fix \((Q,I)\) and \(k\), and let \(Q'=\mu_k(Q)\).  Denote by
\[
\mathfrak A_k(Q,I)
\]
the set of all quadratic monomial ideals \(I'=\langle Z'\rangle\subseteq KQ'\)
arising from Construction~\ref{constr:admissibilityfirst} such that \(I'\) is
admissible.
\end{definition}

\begin{proposition}[Uniqueness implies reversibility]\label{prop:unique-admissibility-reversible}
If \(\mathfrak A_k(Q,I)\) is a singleton, then the admissibility mutation is
reversible at \(k\).  In other words, if
\[
\mathfrak A_k(Q,I)=\{I'\},
\]
then applying Construction~\ref{constr:admissibilityfirst} to \((Q',I')\) at
the same vertex \(k\) recovers the original ideal \(I\).
\end{proposition}

\begin{proof}
The quiver mutation is involutive, so \(\mu_k^2(Q)=Q\).  If the admissible
completion is unique, then there is no ambiguity in the choice of the
permitted skeleton or the cycle-breaking relations.  Hence the reverse
construction is forced to use the same local data, and therefore recovers the
original set of zero-relations.  Thus the original bound quiver \((Q,I)\) is
recovered.
\end{proof}

\begin{remark}
If \(\mathfrak A_k(Q,I)\) has more than one element, then the undecorated
admissibility mutation is generally not reversible, since the construction
forgets which completion was chosen.  In that case, one should work with
decorated bound quivers.
\end{remark}

\begin{definition}[Admissibility certificate]
Fix a bound quiver \((Q,I)\) and a vertex \(k\in Q_0\).  An
\emph{admissibility certificate} at \(k\) is a finite record
\[
\mathcal C_k
\]
containing the local data used in Construction~\ref{constr:admissibilityfirst},
namely:
\begin{enumerate}[label=\textup{(\arabic*)},leftmargin=2.5em]
  \item the forced zero-relations \(Z_{\mathrm{for}}\);
  \item the chosen acyclic permitted skeleton \(P\);
  \item the auxiliary choice used to make \(P\) acyclic, such as a total
  order or a feedback edge set;
  \item the resulting set \(Z'=\mathcal C(Q')\setminus P\) of zero-relations.
\end{enumerate}
\end{definition}

\begin{definition}[Decorated bound quiver]
A \emph{decorated bound quiver} is a triple
\[
(Q,I;\mathcal C),
\]
where \((Q,I)\) is a bound quiver and \(\mathcal C\) is an admissibility
certificate.  We write \(\mathcal C_k\) when the certificate is attached to
mutation at the vertex \(k\).
\end{definition}

\begin{definition}[Reversible admissibility mutation]
Let \((Q,I;\mathcal C)\) be a decorated bound quiver and let
\[
Q'=\mu_k(Q).
\]
A \emph{reversible admissibility mutation} at \(k\) is the assignment
\[
(Q,I;\mathcal C)\longmapsto (Q',I';\mathcal C')
\]
defined by the following steps:
\begin{enumerate}[label=\textbf{Step \arabic*.},leftmargin=3em]
  \item Apply the quiver mutation at \(k\) to obtain \(Q'=\mu_k(Q)\).
  \item Record the forced zero-relations \(Z_{\mathrm{for}}\) exactly as in
  Construction~\ref{constr:admissibilityfirst}.
  \item Choose an acyclic permitted skeleton \(P\subseteq P_{\mathrm{pref}}\)
  and record this choice in the certificate \(\mathcal C_k\).
  \item Define
  \[
    Z'=\mathcal C(Q')\setminus P
    \qquad\text{and}\qquad
    I'=\langle \beta\alpha \mid (\alpha,\beta)\in Z'\rangle\subseteq KQ'.
  \]
  \item Set \(\mathcal C'=(Z_{\mathrm{for}},P,\mathcal C_k)\).
\end{enumerate}
We denote this decorated mutation by
\[
\widetilde{\mu}_k(Q,I;\mathcal C)=(Q',I';\mathcal C').
\]
\end{definition}

\begin{proposition}[Admissibility of the decorated mutation]\label{prop:decorated-admissibility}
Let \((Q,I;\mathcal C)\) be a decorated bound quiver and let
\[
(Q',I';\mathcal C')=\widetilde{\mu}_k(Q,I;\mathcal C)
\]
be obtained by reversible admissibility mutation.  Then \(I'\) is admissible.
\end{proposition}

\begin{proof}
By construction, \(I'\) is generated by length-two paths, so
\[
I'\subseteq \rad_{Q'}^2.
\]
Moreover, the chosen permitted skeleton \(P\) is acyclic.  Therefore the
permitted transition graph \(T(Q',Z')\) has no oriented cycle.  By
Proposition~\ref{prop:admissibility}, this implies that \(I'\) is admissible.
Hence
\[
\rad_{Q'}^m\subseteq I'\subseteq \rad_{Q'}^2
\]
for some \(m\geq 2\).
\end{proof}

\begin{proposition}[Reversibility on decorated bound quivers]\label{prop:decorated-admissibility-reversible}
The reversible admissibility mutation is involutive on decorated bound quivers.
More precisely,
\[
\widetilde{\mu}_k^2(Q,I;\mathcal C)=(Q,I;\mathcal C).
\]
\end{proposition}

\begin{proof}
The quiver mutation is involutive, so \(\mu_k^2(Q)=Q\).  The certificate
\(\mathcal C\) stores the full local choice of forced relations, permitted
skeleton, and cycle-breaking data.  Therefore the second mutation uses the
inverse certificate \(\mathcal C^{-1}\) to recover the same zero-relations and
the same permitted transitions in reverse.  Consequently, the original ideal
\(I\) is recovered.
\end{proof}

\begin{example}[A non-reversible undecorated completion]
Let

\[
Q:\qquad 1\xrightarrow{\alpha}2\xrightarrow{\beta}3,
\qquad
I=\langle \beta\alpha\rangle.
\]
Mutating at \(k=2\) gives the quiver
\begin{equation}\label{}
\nonumber   \xymatrix{
\cdot_{3} \ar[d]_{\beta^*} & \ar[l]_{\eta}\cdot_{1} \\
	\cdot_{2}  \ar[ur]_{\alpha^*}}.
  \end{equation}
The quiver \(Q'\) is an oriented \(3\)-cycle.  There are several admissible
choices of quadratic monomial ideals on \(Q'\), for example
\[
\langle \eta\alpha^*\rangle,\qquad
\langle \beta^*\eta\rangle,\qquad
\langle \alpha^*\beta^*\rangle.
\]
Without a certificate, the construction does not remember which completion was
chosen, so the inverse step is not canonical.

If the certificate records, say, the choice
\[
P=\{(\eta,\beta^*),(\beta^*,\alpha^*)\},
\]
then the resulting ideal is
\[
I'=\langle \eta\alpha^*\rangle.
\]
Applying the inverse decorated mutation at \(k=2\) recovers the original ideal
\[
I=\langle \beta\alpha\rangle.
\]
\end{example}

\begin{remark}
Thus the admissibility-first construction becomes reversible after passing to
decorated bound quivers.  The undecorated construction remains useful for
producing admissible ideals, but the decoration is what makes the mutation
involutive.
\end{remark}

\subsection{A reversible gentle-completion criterion}\label{subsec:reversible-gentle-completion}

Let \((Q,I)\) be a bound quiver and let \(Q'=\mu_k(Q)\).  Assume that the
admissibility-first procedure of Construction~\ref{constr:admissibilityfirst}
produces at least one gentle completion on \(Q'\).  For a given choice of local
matching and cycle-breaking relations, let
\[
I'=\langle Z'\rangle\subseteq KQ'
\]
be the resulting quadratic monomial ideal.  We say that \(I'\) is a
\emph{gentle completion} of \((Q,I)\) at \(k\) if \(KQ'/I'\) is gentle.

\begin{definition}[The set of gentle completions]
Fix \((Q,I)\) and \(k\).  Denote by
\[
\mathfrak G_k(Q,I)
\]
the set of all quadratic monomial ideals \(I'=\langle Z'\rangle\subseteq KQ'\)
arising from Construction~\ref{constr:admissibilityfirst} such that
\(KQ'/I'\) is gentle.
\end{definition}

\begin{proposition}[Uniqueness implies reversibility]\label{prop:uniqueness-reversible}
If \(\mathfrak G_k(Q,I)\) is a singleton, then the gentle-completion criterion
is reversible at \(k\).  In other words, if
\[
\mathfrak G_k(Q,I)=\{I'\},
\]
then applying the same construction to \((Q',I')\) at the same vertex \(k\)
recovers the original ideal \(I\).
\end{proposition}

\begin{proof}
The quiver mutation is involutive, so \(\mu_k^2(Q)=Q\).  If the gentle
completion is unique, then there is no ambiguity in the choice of local
matching or cycle-breaking relations.  Hence the reverse construction is forced
to use the same local data, and therefore recovers the original set of
zero-relations.  Thus the original bound quiver \((Q,I)\) is recovered.
\end{proof}

\begin{remark}
If \(\mathfrak G_k(Q,I)\) has more than one element, then the undecorated
gentle completion is generally not reversible, since the mutation process
forgets which completion was chosen.
\end{remark}

\begin{definition}[Decorated gentle completion]
To obtain a reversible theory in the non-unique case, we enlarge the objects to
decorated bound quivers
\[
(Q,I;\mathcal C),
\]
where \(\mathcal C\) is a certificate recording:
\begin{enumerate}[label=\textup{(\arabic*)},leftmargin=2.5em]
  \item the chosen local matching at each affected vertex;
  \item the cycle-breaking zero-relations inserted by the construction;
  \item the forced transported relations arising from the mutation at \(k\).
\end{enumerate}
\end{definition}

\begin{construction}[Reversible gentle mutation]\label{constr:reversible-gentle-mutation}
Let \((Q,I;\mathcal C)\) be a decorated bound quiver and let \(Q'=\mu_k(Q)\).
Choose a gentle completion \(I'=\langle Z'\rangle\subseteq KQ'\) using
Construction~\ref{constr:admissibilityfirst}, and let \(\mathcal C'\) be the
certificate recording all local choices made in the completion.

Define
\[
\widetilde{\mu}_k(Q,I;\mathcal C)=(Q',I';\mathcal C').
\]
The inverse mutation is defined by using the inverse certificate
\(\mathcal C'^{-1}\), which restores the original local choices when the same
vertex \(k\) is mutated again.
\end{construction}

\begin{proposition}[Involutivity on decorated objects]\label{prop:decorated-gentle-involutive}
On decorated bound quivers, the reversible gentle mutation is involutive:
\[
\widetilde{\mu}_k^2(Q,I;\mathcal C)=(Q,I;\mathcal C).
\]
\end{proposition}

\begin{proof}
The quiver part is involutive, since \(\mu_k^2(Q)=Q\).  The certificate
\(\mathcal C\) stores the local matching and cycle-breaking choices used to
construct \(I'\).  Therefore, when the construction is applied again at the
same vertex using the inverse certificate, every transported relation is sent
back to its original position, every inserted cycle-breaking relation is
removed, and the original ideal \(I\) is recovered.
\end{proof}

\begin{example}[Non-reversible without a certificate]
Let
\[
Q:\qquad 1\xrightarrow{\alpha}2\xrightarrow{\beta}3,
\qquad
I=\langle \beta\alpha\rangle.
\]
Mutating at \(k=2\) gives the quiver

\begin{equation}\label{}
\nonumber Q'=\mu_2(Q):  \xymatrix{
\cdot_{3} \ar[d]_{\beta^*} & \ar[l]_{\eta}\cdot_{1} \\
	\cdot_{2}  \ar[ur]_{\alpha^*}}.
  \end{equation}

The quiver \(Q'\) is a triangle, and there are several gentle completions on
\(Q'\), for example
\[
\langle \eta\alpha^*\rangle,\qquad
\langle \beta^*\eta\rangle,\qquad
\langle \alpha^*\beta^*\rangle.
\]
Thus \(\mathfrak G_2(Q,I)\) is not a singleton, and the undecorated gentle
completion is not reversible.

If, however, the certificate records that the chosen completion is
\(\langle \eta\alpha^*\rangle\), then the inverse decorated mutation at \(2\)
recovers \(\langle \beta\alpha\rangle\).
\end{example}

\begin{remark}
Therefore, the right reversible formulation is:

\[
\text{undecorated gentle completion is reversible only when the completion is unique;}
\]
otherwise, one should work with decorated bound quivers
\((Q,I;\mathcal C)\), for which the mutation is involutive.
\end{remark}

\section{Classes for which construction~\ref{constr:gentle-admissible-completion} produces gentle algebras}
\label{sec:classes-gentle-relation-mutation}

In this section we record several useful classes of mutations for which the
relation-mutation construction produces a gentle algebra, and we also describe
classes where no such conclusion is possible.  Throughout, let
\((Q,I)\) be a gentle bound quiver with
\[
  I=\langle Z\rangle \subseteq \K Q,
\]
where \(Z\) is a set of zero length-two paths.  Let
\[
  Q'=\mu_k(Q)
\]
be the quiver obtained by mutating \(Q\) at a vertex \(k\).  We use the
notation and terminology of Construction~\ref{constr:admissibilityfirst} and
Construction ~\ref{constr:gentle-admissible-completion}.  Thus
\(Z_{\mathrm{for}}\) denotes the forced zero-relations after mutation,
\(P\) denotes the chosen set of permitted transitions in \(Q'\), and
\[
  Z'=\mathcal C(Q')\setminus P,
  \qquad
  I'=\langle Z'\rangle\subseteq \K Q'.
\]

The admissibility-first construction guarantees admissibility by choosing the
permitted transitions to be acyclic.  However, gentleness is stronger than
admissibility.  In addition to admissibility, we must require the mutated quiver
\(Q'\) to satisfy the degree condition and the permitted and forbidden
continuation conditions at every arrow.

\subsection{The general criterion}

We begin with a compact criterion describing exactly what the construction must
achieve in order to produce a gentle algebra.

\begin{definition}[Gentle-completable mutation datum]
\label{def:gentle-completable-mutation-datum}
Let \((Q,I)\) be a gentle bound quiver and let \(Q'=\mu_k(Q)\).  We say that
\((Q,I,k)\) is \emph{gentle-completable} if the following conditions hold.
\begin{enumerate}[label=\textup{(G\arabic*)},leftmargin=3em]
  \item The mutated quiver \(Q'\) is degree-gentle, that is,
  \[
    \deg_{Q'}^+(v)\leq 2
    \qquad\text{and}\qquad
    \deg_{Q'}^-(v)\leq 2
  \]
  for every vertex \(v\in Q'_0\).

  \item The forced zero-relations \(Z_{\mathrm{for}}\) are locally compatible
  with a matching choice.  Equivalently, for every vertex \(v\in Q'_0\), if
  \[
    \mathcal C_v(Q')=
    \operatorname{In}_{Q'}(v)\times \operatorname{Out}_{Q'}(v),
  \]
  then there exists a subset
  \[
    P_v\subseteq \mathcal C_v(Q')\setminus Z_{\mathrm{for}}
  \]
  satisfying the local matching table of
  Definition~\ref{def:acyclic-local-matching}.

  \item The local choices \(P_v\) can be chosen globally so that
  \[
    P=\bigcup_{v\in Q'_0}P_v
  \]
  is acyclic as a permitted transition graph.
\end{enumerate}
\end{definition}

\begin{proposition}[Gentle-completion criterion]
\label{prop:gentle-completion-criterion}
Let \((Q,I)\) be a gentle bound quiver and let \(Q'=\mu_k(Q)\).  If
\((Q,I,k)\) is gentle-completable, then Construction~\ref{constr:gentle-admissible-completion}
produces a quadratic monomial ideal
\[
  I'=\langle Z'\rangle\subseteq \K Q'
\]
such that \(\K Q'/I'\) is a gentle algebra.
\end{proposition}

\begin{proof}
By assumption, \(Q'\) is degree-gentle.  By Construction~\ref{constr:gentle-admissible-completion},
we choose an acyclic local matching
\[
  P\subseteq \mathcal C(Q')\setminus Z_{\mathrm{for}}.
\]
Then
\[
  Z'=\mathcal C(Q')\setminus P
\]
and
\[
  I'=\langle \beta\alpha\mid (\alpha,\beta)\in Z'\rangle
\]
is generated by length-two paths.  Since \(P\) is acyclic, the permitted
transition graph \(T(Q',Z')\) has no oriented cycle.  Hence
\(I'\) is admissible by the admissibility criterion.

It remains only to check the permitted and forbidden continuation conditions.
The local matching condition ensures that for each arrow \(\alpha\), there is
at most one arrow \(\beta\) such that \((\alpha,\beta)\in P\), and at most one
arrow \(\gamma\) such that \((\gamma,\alpha)\in P\).  Thus \(\alpha\) has at most
one permitted right continuation and at most one permitted left continuation.
Since \(Z'=\mathcal C(Q')\setminus P\), the same local matching condition also
implies that \(\alpha\) has at most one zero right continuation and at most one
zero left continuation.  Therefore \(\K Q'/I'\) is gentle.
\end{proof}

The following obstruction is immediate but important.

\begin{proposition}[Obstructions to gentleness]
\label{prop:obstructions-gentleness}
Let \((Q,I)\) be a gentle bound quiver and let \(Q'=\mu_k(Q)\).  Then
Construction~\ref{constr:gentle-admissible-completion} cannot produce a gentle
algebra in any of the following cases.
\begin{enumerate}[label=\textup{(O\arabic*)},leftmargin=3em]
  \item The mutated quiver \(Q'\) is not degree-gentle.
  \item At some vertex \(v\in Q'_0\), the forced zero-relations
  \(Z_{\mathrm{for}}\) do not allow a local matching.
  \item Every local matching compatible with \(Z_{\mathrm{for}}\) produces a
  permitted transition graph with an oriented cycle.
\end{enumerate}
\end{proposition}

\begin{proof}
If \(Q'\) is not degree-gentle, then the first condition in the definition of a
gentle algebra already fails, independently of the choice of ideal \(I'\).  This
proves (O1).

For (O2), recall that the gentle conditions force the permitted transitions and
the zero transitions at each vertex to be matchings.  If the forced
zero-relations make such a matching impossible at some vertex, then no
completion compatible with the forced part of the construction can be gentle.

For (O3), if every compatible local matching leaves an oriented cycle in the
permitted transition graph, then the resulting ideal is not admissible.  Hence
it cannot define a gentle algebra.
\end{proof}

\subsection{Positive classes}

The following classes give useful situations in which the construction can be
completed successfully.

\begin{proposition}[Acyclic target quivers]
\label{prop:acyclic-target-gentle}
Let \((Q,I)\) be a gentle bound quiver and let \(Q'=\mu_k(Q)\).  Suppose that
\(Q'\) is acyclic and degree-gentle.  If the forced zero-relations
\(Z_{\mathrm{for}}\) are locally compatible with the matching table of
Definition~\ref{def:acyclic-local-matching}, then
Construction~\ref{constr:gentle-admissible-completion} can be completed so that
\[
  \K Q'/I'
\]
is gentle.
\end{proposition}

\begin{proof}
Since \(Q'\) is acyclic, every permitted transition graph on arrows of \(Q'\) is
acyclic: an oriented cycle in the permitted transition graph would determine an
oriented cycle in \(Q'\).  Thus, after choosing local matchings compatible with
\(Z_{\mathrm{for}}\), the global permitted transition graph is automatically
acyclic.  The result follows from Proposition~\ref{prop:gentle-completion-criterion}.
\end{proof}

\begin{corollary}[Controlled source and sink mutations]
\label{cor:controlled-source-sink}
Let \((Q,I)\) be a gentle bound quiver with \(Q\) acyclic.
\begin{enumerate}[label=\textup{(\alph*)},leftmargin=3em]
  \item Suppose that \(k\) is a source and that
  \[
    \deg_Q^+(j)\leq 1
  \]
  for every arrow \(k\to j\).  Then \(Q'=\mu_k(Q)\) is acyclic and
  degree-gentle.  If the forced zero-relations are locally compatible, then
  \(\K Q'/I'\) is gentle.

  \item Suppose that \(k\) is a sink and that
  \[
    \deg_Q^-(i)\leq 1
  \]
  for every arrow \(i\to k\).  Then \(Q'=\mu_k(Q)\) is acyclic and
  degree-gentle.  If the forced zero-relations are locally compatible, then
  \(\K Q'/I'\) is gentle.
\end{enumerate}
\end{corollary}

\begin{proof}
If \(k\) is a source, then no new composite arrows are created by mutation at
\(k\); mutation only reverses the arrows leaving \(k\).  Since \(Q\) is acyclic,
\(Q'\) is again acyclic.  The condition \(\deg_Q^+(j)\leq 1\) for every
\(k\to j\) is precisely what prevents a neighbor \(j\) from acquiring more than
two outgoing arrows after reversal.  Thus \(Q'\) is degree-gentle.  The result
then follows from Proposition~\ref{prop:acyclic-target-gentle}.  The sink case
is dual.
\end{proof}

\begin{proposition}[Underlying graph of type \(A_n\)]
\label{prop:An-class-gentle}
Let \((Q,I)\) be a gentle bound quiver whose underlying unoriented graph is a
path of type \(A_n\).  Let \(k\in Q_0\) and let \(Q'=\mu_k(Q)\).  Then the
relation-mutation construction can be completed so that
\[
  I'=\langle Z'\rangle\subseteq \K Q'
\]
is admissible and \(\K Q'/I'\) is gentle.
\end{proposition}

\begin{proof}
Since the underlying graph is a path, the mutation at \(k\) is local and has no
branching.  If \(k\) is an endpoint, then mutation only reverses the unique
arrow incident with \(k\), so the mutated quiver is again acyclic and
path-shaped.  If \(k\) is an interior source or sink, mutation again only
reverses the two arrows incident with \(k\), so the target quiver is acyclic.
In these cases the claim follows from Proposition~\ref{prop:acyclic-target-gentle}.

It remains to consider the case where \(k\) has one incoming and one outgoing
arrow,
\[
  i\xrightarrow{\alpha} k\xrightarrow{\beta} j.
\]
Then mutation creates a single composite arrow
\[
  \eta:i\to j
\]
and reverses the two arrows incident with \(k\):
\[
  \alpha^*:k\to i,
  \qquad
  \beta^*:j\to k.
\]
Thus the only oriented cycle created by the mutation is the triangle

\begin{equation}\label{}
\nonumber   \xymatrix{
\cdot_{i} \ar[r]^{\eta} & \ar[dl]^{\beta^*}\cdot_{j}   \\
\cdot_{k}\ar[u]^{\alpha^*}}.
  \end{equation}

Because there is no branching, each affected vertex has at most two incoming
and at most two outgoing arrows.  Hence \(Q'\) is degree-gentle.

We now choose the local matching so that at least one consecutive pair on this
triangle is declared zero.  This breaks the only possible permitted oriented
cycle.  The remaining local choices can be made according to the local matching
table, because each affected vertex has either a \((1,1)\)-, \((1,2)\)-, or
\((2,1)\)-configuration.  Therefore the resulting permitted transition graph is
acyclic, and Proposition~\ref{prop:gentle-completion-criterion} applies.
\end{proof}

\begin{example}[A type \(A_3\) mutation]
\label{ex:A3-positive}
Let
\[
  Q:
  \qquad
  1\xrightarrow{\alpha}2\xrightarrow{\beta}3
\]
and let
\[
  I=\langle \beta\alpha\rangle.
\]
Then \((Q,I)\) is gentle.  Mutate at \(k=2\).  The mutated quiver has arrows
\[
  \alpha^*:2\to 1,
  \qquad
  \eta:1\to 3,
  \qquad
  \beta^*:3\to 2,
\]
so \(Q'=\mu_2(Q)\) is the oriented triangle

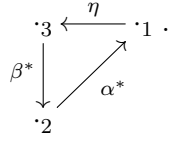
\begin{equation}\label{}
\nonumber   \xymatrix{
\cdot_{3} \ar[d]_{\beta^*} & \ar[l]_{\eta}\cdot_{1} \\
	\cdot_{2}  \ar[ur]_{\alpha^*}}.
  \end{equation}

Since the initial path \(\beta\alpha\) was zero, the complement-transpose rule
allows the reversed path
\[
  3\xrightarrow{\beta^*}2\xrightarrow{\alpha^*}1
\]
to remain permitted.  To make the ideal admissible and gentle, we break the
triangle by declaring, for instance,
\[
  \eta\alpha^*=0.
\]
Thus we may take
\[
  I'=\langle \eta\alpha^*\rangle\subseteq \K Q'.
\]
The permitted transition graph is the directed path
\[
  \eta\longrightarrow \beta^*\longrightarrow \alpha^*,
\]
together with no additional cycle.  Hence \(I'\) is admissible, and
\(\K Q'/I'\) is gentle.
\end{example}

\begin{proposition}[Cycle-breaking degree-gentle mutations]
\label{prop:cycle-breaking-gentle}
Let \((Q,I)\) be a gentle bound quiver and let \(Q'=\mu_k(Q)\).  Suppose that
\(Q'\) is degree-gentle and that the forced zero-relations are locally
compatible.  If the local matching choices can be made so that every oriented
cycle in \(Q'\) contains a consecutive length-two subpath belonging to \(Z'\),
then \(\K Q'/I'\) is gentle.
\end{proposition}

\begin{proof}
The condition says exactly that no oriented cycle of \(Q'\) is fully permitted.
Equivalently, the permitted transition graph \(T(Q',Z')\) has no oriented
cycle.  Together with degree-gentleness and the local matching condition, this
is precisely the hypothesis of Proposition~\ref{prop:gentle-completion-criterion}.
\end{proof}

\subsection{Dynkin and finite mutation classes}

The preceding positive results should not be interpreted as saying that finite
representation type, Dynkin type, finite cluster type, or finite mutation type is
sufficient by itself.  Those are mutation-theoretic or representation-theoretic
properties of the quiver, whereas gentleness of \(\K Q'/I'\) depends on three
local pieces of data:
\[
  \text{degree-gentleness of }Q',
  \qquad
  \text{local compatibility of }Z_{\mathrm{for}},
  \qquad
  \text{acyclicity of the permitted transitions.}
\]

For ordinary simply-laced Dynkin quivers, the path case \(A_n\) is covered by
Proposition~\ref{prop:An-class-gentle}.  However, Dynkin types with branching,
such as \(D_n\) and \(E_n\), are not automatically preserved by relation
mutation.  Mutation at or near a branching vertex may destroy degree-gentleness,
and even when degree-gentleness survives, the forced relations may make the
local matching condition impossible.

\begin{example}[Dynkin type \(D_4\): failure of degree-gentleness]
\label{ex:D4-degree-failure}
Let \(Q\) be the quiver

\begin{equation}\label{}
\nonumber   \xymatrix{
\cdot_{1} \ar[r]^{\alpha} & \ar[r]^{\beta}\cdot_{2}\ar[d]^{\gamma} & \cdot_{3}  \\
&\cdot_{4}&}.
  \end{equation}

and let
\[
  I=\langle \beta\alpha\rangle.
\]
Then \((Q,I)\) is gentle.  Indeed, vertex \(2\) has one incoming arrow and two
outgoing arrows, and among the two length-two paths through \(2\), one is zero
and one is permitted:
\[
  \beta\alpha\in I,
  \qquad
  \gamma\alpha\notin I.
\]

The underlying unoriented graph is Dynkin type \(D_4\).  Now mutate at the
source vertex \(k=1\).  Since \(1\) is a source, mutation creates no new
composite arrows; it only reverses the arrow \(1\to 2\).  Hence
\(Q'=\mu_1(Q)\) is
\[
  2\to 1,
  \qquad
  2\to 3,
  \qquad
  2\to 4.
\]
Thus vertex \(2\) has three outgoing arrows in \(Q'\).  Therefore \(Q'\) is not
degree-gentle.  Consequently, no choice of ideal \(I'\subseteq \K Q'\) can make
\(\K Q'/I'\) gentle.  This shows that Dynkin type, and hence finite cluster type
in this example, is not sufficient to guarantee success of the relation-mutation
procedure.
\end{example}

\begin{example}[Dynkin type \(D_5\): failure of local compatibility]
\label{ex:D5-local-failure}
Let \(Q\) be the quiver
\[
  0\xrightarrow{\rho}1\xrightarrow{\alpha}2,
  \qquad
  2\xrightarrow{\beta}3,
  \qquad
  2\xrightarrow{\gamma}4,
\]
and set
\[
  I=\langle \alpha\rho,\ \beta\alpha\rangle.
\]
Then \((Q,I)\) is gentle.  The zero-relation \(\alpha\rho=0\) occurs through
vertex \(1\), while at vertex \(2\) the two continuations of \(\alpha\) satisfy
\[
  \beta\alpha\in I,
  \qquad
  \gamma\alpha\notin I.
\]
Thus each arrow has at most one permitted and at most one forbidden
continuation on either side.

Mutate at \(k=2\).  The paths
\[
  1\xrightarrow{\alpha}2\xrightarrow{\beta}3,
  \qquad
  1\xrightarrow{\alpha}2\xrightarrow{\gamma}4
\]
produce new composite arrows
\[
  \eta:1\to 3,
  \qquad
  \xi:1\to 4.
\]
The arrows incident with \(2\) are reversed:
\[
  \alpha^*:2\to 1,
  \qquad
  \beta^*:3\to 2,
  \qquad
  \gamma^*:4\to 2.
\]
Thus \(Q'=\mu_2(Q)\) is degree-gentle.  For example, vertex \(1\) has two
incoming arrows, \(\rho\) and \(\alpha^*\), and two outgoing arrows, \(\eta\)
and \(\xi\).

However, the forced relation-transport step creates an obstruction.  Since
\[
  \alpha\rho\in I,
\]
the transport rule forces both
\[
  \eta\rho\in I'
  \qquad\text{and}\qquad
  \xi\rho\in I'.
\]
Thus, at vertex \(1\), the incoming arrow \(\rho\) has two forced zero right
continuations:
\[
  \rho\longrightarrow \eta,
  \qquad
  \rho\longrightarrow \xi.
\]
This violates the local matching condition.  Equivalently, in the
\((2,2)\)-configuration at vertex \(1\), any permitted local matching would need
to choose one permitted continuation starting at \(\rho\), but both possible
continuations from \(\rho\) are already forced to be zero.

Therefore the forced zero-relations are not locally compatible with gentleness.
Although the admissibility-first construction can still produce an admissible
quadratic monomial ideal by killing enough length-two paths, it cannot produce
a gentle algebra compatible with these forced transported relations.
\end{example}

\begin{remark}[Finite mutation type is not enough]
\label{rem:finite-mutation-not-enough}
Finite mutation type is a property of the quiver mutation class.  It does not
record the admissible ideal or the local permitted and forbidden continuations.
Thus it cannot by itself guarantee that relation mutation produces a gentle
algebra.  The examples above already show this: Dynkin quivers are of finite
cluster type, hence of finite mutation type, but the relation-mutation procedure
may fail either because the mutated quiver is not degree-gentle
(Example~\ref{ex:D4-degree-failure}) or because the forced zero-relations cannot
be completed to a local matching (Example~\ref{ex:D5-local-failure}).
\end{remark}

\subsection{A practical classification summary}

The preceding discussion can be summarized as follows.

\begin{theorem}[Practical success and failure criteria]
\label{thm:practical-success-failure}
Let \((Q,I)\) be a gentle bound quiver and let \(Q'=\mu_k(Q)\).  Let
\(Z_{\mathrm{for}}\) be the forced zero-relation set determined by the
relation-mutation construction.

\begin{enumerate}[label=\textup{(\arabic*)},leftmargin=3em]
  \item If \(Q'\) is acyclic, degree-gentle, and \(Z_{\mathrm{for}}\) is locally
  compatible with a matching, then the construction produces a gentle algebra
  \(\K Q'/I'\).

  \item If the underlying unoriented graph of \(Q\) is a path of type \(A_n\),
  then the construction can be completed so that \(\K Q'/I'\) is gentle.

  \item More generally, if \(Q'\) is degree-gentle and the forced
  zero-relations admit an acyclic local matching, then the construction
  produces a gentle algebra \(\K Q'/I'\).

  \item If \(Q'\) is not degree-gentle, then no ideal \(I'\subseteq \K Q'\) can
  make \(\K Q'/I'\) gentle.

  \item If the forced zero-relations do not admit a local matching at some
  vertex, then the construction cannot produce a gentle algebra compatible
  with those forced relations.

  \item Dynkin type, finite cluster type, and finite mutation type do not by
  themselves guarantee success.  They must be supplemented by the local degree,
  local matching, and acyclicity conditions above.
\end{enumerate}
\end{theorem}

\begin{proof}
Parts (1)--(3) follow from Propositions~\ref{prop:acyclic-target-gentle},
\ref{prop:An-class-gentle}, and \ref{prop:gentle-completion-criterion}.  Parts
(4) and (5) follow from Proposition~\ref{prop:obstructions-gentleness}.  Part
(6) follows from Examples~\ref{ex:D4-degree-failure} and
\ref{ex:D5-local-failure}.
\end{proof}

\section{Possible relations between the initial and mutated gentle algebras}
\label{sec:comparison-mutated-algebras}

Let \((Q,I)\) be a gentle bound quiver and let
\[
  Q'=\mu_k(Q)
\]
be the quiver obtained by mutating \(Q\) at the vertex \(k\).  Suppose that
\(Q'\) is degree-gentle.  Let \(Z'\) be the set of quadratic zero-relations
constructed by 
Construction ~\ref{constr:gentle-admissible-completion}, and set
\[
  I'=\langle Z'\rangle\subseteq \K Q'.
\]
Assume further that the local matching condition and the permitted-cycle
condition hold, so that \(\K Q'/I'\) is again gentle by
Proposition ~\ref{prop:gentle-completion-criterion}.  In this section we clarify what kind of
relationship one should expect between
\[
  A=\K Q/I
  \qquad\text{and}\qquad
  B=\K Q'/I'.
\]

The main point is that the constructions in this paper are mutation
constructions, not isomorphism constructions.  They provide a way to transport
or reconstruct quadratic zero-relations after quiver mutation, but they do not
imply that the original and mutated algebras are isomorphic.  In general, they
need not even be derived equivalent.

\subsection{Isomorphism is exceptional}

The strongest possible relation between \(A\) and \(B\) is an algebra
isomorphism.  A natural sufficient condition for this is an isomorphism of
bound quivers.  Namely, suppose that there is a quiver isomorphism
\[
  \sigma:Q\longrightarrow Q'
\]
such that the induced path-algebra isomorphism
\[
  \K Q\longrightarrow \K Q'
\]
sends \(I\) onto \(I'\).  Equivalently, in the quadratic monomial case,
\[
  \beta\alpha\in I
  \quad\Longleftrightarrow\quad
  \sigma(\beta)\sigma(\alpha)\in I'
\]
for every composable pair of arrows \(\alpha,\beta\) in \(Q\).  Then
\[
  \K Q/I\cong \K Q'/I'.
\]

Conversely, under the usual basic presentation assumptions, the Gabriel quiver
of a finite-dimensional basic algebra is recovered from the algebra.  Thus, if
the directed quivers \(Q\) and \(Q'\) are not isomorphic, then one should not
expect \(\K Q/I\) and \(\K Q'/I'\) to be isomorphic.  Therefore, in the present
setting, isomorphism should be regarded as exceptional:
\[
  \boxed{
  \K Q/I \cong \K Q'/I'
  \quad\text{only in special cases where}\quad
  (Q,I)\cong (Q',I')
  \text{ as bound quivers.}
  }
\]

There is also an opposite-algebra possibility.  If
\[
  (Q',I')\cong (Q^{\operatorname{op}},I^{\operatorname{op}}),
\]
then
\[
  \K Q'/I'\cong (\K Q/I)^{\operatorname{op}}.
\]
This is weaker than being isomorphic to \(\K Q/I\) itself, unless the algebra
is isomorphic to its opposite.

\begin{example}[A mutation-related pair that is not isomorphic]
Let
\[
  Q:
  \qquad
  1\xrightarrow{\alpha}2\xrightarrow{\beta}3,
  \qquad
  I=0.
\]
Then \(\K Q/I\) is gentle. Mutating at \(k=2\) produces the quiver described in
Example~\ref{ex:pathzero} below. Since the original and mutated bound quivers
are not isomorphic, their path algebras are not isomorphic either.
\end{example}

\subsection{The mutation relation provided by the construction}

Since isomorphism is too strong, the relation directly provided by
Construction~\ref{constr:gentle-admissible-completion} is a mutation relation
between bound quivers.

\begin{definition}
Let \((Q,I)\) and \((Q',I')\) be gentle bound quivers. We say that they are
\emph{elementarily gentle-mutation-related at \(k\)} if
\[
  Q'=\mu_k(Q)
\]
and \(I'\) is obtained from \(I\) by Construction~\ref{constr:gentle-admissible-completion},
with the local matching and admissibility conditions satisfied. We write
\[
  (Q,I)\sim_{\mathrm{gmut},k}(Q',I').
\]
The equivalence relation generated by elementary gentle mutations will be
denoted by
\[
  \sim_{\mathrm{gmut}}.
\]
\end{definition}

Thus, whenever the hypotheses of Proposition~\ref{prop:gentle-completion-criterion} are satisfied,
the construction guarantees
\[
  (Q,I)\sim_{\mathrm{gmut},k}(Q',I').
\]
This is a combinatorial and local relation. It records that \(B\) is obtained
from \(A\) by mutating the underlying quiver and by choosing compatible
quadratic zero-relations on the mutated quiver.

At the cluster level, \(Q\) and \(Q'\) lie in the same mutation class. Since
rooted mutation loops are defined relative to a chosen initial seed, the
rooted-mutation-group discussion should be read as a comparison inside that
same mutation class rather than as a new invariant attached to the bound
quiver. For a detailed account of rooted mutation groups, see
\cite{SalehRootedMutationGroups}.

\subsection{Derived equivalence: possible but not automatic}

A more flexible relation than algebra isomorphism is derived equivalence:
\[
  D^b(\K Q/I)\simeq D^b(\K Q'/I').
\]
Construction~\ref{constr:gentle-admissible-completion} does not automatically imply such
an equivalence. Nevertheless, derived equivalence is a natural relation to
look for when the local relation mutation agrees with a tilting mutation of the
algebra.

Let \(A=\K Q/I\). Denote by \(T_k^-(A)\) and \(T_k^+(A)\) the standard
negative and positive algebra-mutation complexes at the vertex \(k\). If one
of these complexes is a tilting complex and
\[
  B=\K Q'/I'
  \cong
  \operatorname{End}_{D^b(A)}\bigl(T_k^\pm(A)\bigr),
\]
then Rickard's Morita theory for derived categories gives
\[
  D^b(A)\simeq D^b(B).
\]
In this case, Construction~\ref{constr:gentle-admissible-completion} is not merely a
combinatorial mutation of relations; it coincides with a genuine algebra
mutation in the sense of tilting theory \cite{LadkaniMutation}.

\begin{proposition}[Derived equivalence criterion]
Let \((Q,I)\) be gentle, let \(Q'=\mu_k(Q)\), and let \(I'\) be obtained from
Construction~\ref{constr:gentle-admissible-completion}. Set
\[
  A=\K Q/I,
  \qquad
  B=\K Q'/I'.
\]
If
\[
  B\cong \operatorname{End}_{D^b(A)}\bigl(T_k^\pm(A)\bigr)
\]
for a valid positive or negative algebra-mutation complex \(T_k^\pm(A)\), then
\[
  D^b(A)\simeq D^b(B).
\]
\end{proposition}

\begin{proof}
If \(T_k^\pm(A)\) is a tilting complex, then its endomorphism algebra in
\(D^b(A)\) is derived equivalent to \(A\). Since this endomorphism algebra is
assumed to be isomorphic to \(B\), the conclusion follows.
\end{proof}

\begin{remark}
The proposition gives a sufficient condition, not a necessary one. In
practice, one may first construct \(I'\) by the local matching rule and then
compare \(\K Q'/I'\) with the endomorphism algebra produced by a positive or
negative algebra mutation. If the two agree, the mutated gentle algebra is
derived equivalent to the original one.
\end{remark}

\subsection{Derived obstructions}

Derived equivalence can fail even when both algebras are gentle. For gentle
algebras, the Avella-Alaminos--Geiss invariant is a derived invariant. Hence,
if
\[
  \phi_{\K Q/I}\neq \phi_{\K Q'/I'},
\]
where \(\phi\) denotes the Avella-Alaminos--Geiss invariant, then the two
gentle algebras are not derived equivalent.

The Cartan determinant gives another quick obstruction. If \(C_A\) and
\(C_B\) are the Cartan matrices of \(A\) and \(B\), derived equivalence implies
\[
  |\det C_A|=|\det C_B|.
\]
Thus a difference in Cartan determinants also rules out derived equivalence.

In particular, the number of arrows by itself should not be used as a derived
invariant; one should instead compare invariants such as \(\phi\) or the Cartan
determinant.

\begin{example}[Derived obstruction by invariants]
Whenever the computed Avella-Alaminos--Geiss invariant or the Cartan determinant
differs between two gentle algebras, the algebras cannot be derived equivalent.
This is the correct way to rule out derived equivalence in this setting.
\end{example}

\begin{example}[Not isomorphic but derived equivalent]
Let
\[
  Q:
  \qquad
  2\longleftarrow 1\longrightarrow 3,
  \qquad
  I=0.
\]
Mutate at the source \(k=1\). Since \(1\) is a source, no composite arrows are
created; the two arrows are simply reversed. Hence
\[
  Q'=\mu_1(Q):
  \qquad
  2\longrightarrow 1\longleftarrow 3,
  \qquad
  I'=0.
\]
The two path algebras are not isomorphic as bound quiver algebras, since the
directed quivers are not isomorphic. However, they are hereditary algebras of
the same underlying Dynkin type \(A_3\), and they are derived equivalent by the
classical reflection/tilting construction \cite{BernsteinGelfandPonomarev}.
This illustrates the typical intermediate situation:
\[
  \K Q/I\not\cong \K Q'/I',
  \qquad
  D^b(\K Q/I)\simeq D^b(\K Q'/I').
\]
\end{example}

\subsection{Summary of possible relations}

The preceding discussion gives the following hierarchy.
\[
\boxed{
\begin{array}{c}
(Q,I)\cong (Q',I')\text{ as bound quivers}
\end{array}
\Longrightarrow
\begin{array}{c}
\K Q/I\cong \K Q'/I'
\end{array}}
\]

\[
\boxed{
\begin{array}{c}
\K Q'/I'\cong \operatorname{End}_{D^b(\K Q/I)}(T_k^\pm)
\end{array}
\Longrightarrow
\begin{array}{c}
D^b(\K Q/I)\simeq D^b(\K Q'/I')
\end{array}}
\]

\[
\boxed{
\begin{array}{c}
I'\text{ is obtained from }I\text{ by Construction~}
\ref{constr:gentle-admissible-completion}
\end{array}
\Longrightarrow
\begin{array}{c}
(Q,I)\sim_{\mathrm{gmut}}(Q',I')
\end{array}}
\]

Thus the construction always gives a mutation relation between gentle bound
quivers, provided the local and global tests are satisfied. It gives an
isomorphism only in exceptional symmetric cases, and it gives a derived
equivalence only when the constructed algebra agrees with a genuine tilting or
algebra mutation, or when derived-equivalence invariants confirm that such an
equivalence is possible.

\section{Examples}
\label{sec:examples}

This section illustrates the construction and the comparison results above.
The first examples show successful relation mutations.  The non-examples show
why the degree-gentle hypothesis and the admissibility condition cannot be
omitted.

\begin{example}[Mutating a path with no initial zero-relation]
\label{ex:pathzero}
Let
\[
  Q:
  \qquad
  1\xrightarrow{\alpha}2\xrightarrow{\beta}3,
  \qquad
  I=0.
\]
The quiver is acyclic, so \(I=0\) is admissible, and \(\K Q/I\) is gentle.
Mutate at \(k=2\).  The mutated quiver \(Q'=\mu_2(Q)\) has arrows

\[
  \alpha^*:2\to 1,
  \qquad
  \beta^*:3\to 2,
  \qquad
  \eta:1\to 3,
\]
where \(\eta\) is the composite arrow coming from the initial path
\[
  1\xrightarrow{\alpha}2\xrightarrow{\beta}3.
\]
Hence \(Q'\) is the oriented cycle

\begin{equation}\label{}
\nonumber   \xymatrix{
\cdot_{1} \ar[r]^{\eta} & \ar[dl]^{\beta^*}\cdot_{3}   \\
\cdot_{2}\ar[u]^{\alpha^*}}.
  \end{equation}

Since the initial path \(\beta\alpha\) was permitted, the complement-transpose rule
kills the reversed path through the mutated vertex:
\[
  \alpha^*\beta^*\in I'.
\]
Thus one may take
\[
  I'=\langle \alpha^*\beta^*\rangle.
\]
The unique oriented cycle contains a zero-relation, so \(I'\) is admissible.
The local gentle conditions are automatic because every vertex has one incoming
and one outgoing arrow.  Hence \(\K Q'/I'\) is gentle.
\end{example}

\begin{example}[Mutating a path whose length-two path is killed]
\label{ex:killed-path}
Let
\[
  Q:
  \qquad
  1\xrightarrow{\alpha}2\xrightarrow{\beta}3,
  \qquad
  I=\langle \beta\alpha\rangle.
\]
Then \(\K Q/I\) is gentle.  Mutating again at \(k=2\) gives the same oriented
cycle as in Example~\ref{ex:pathzero}.  Since the initial path \(\beta\alpha\) was
zero, the complement-transpose rule makes the reversed path
\(\alpha^*\beta^*\) permitted.  However, the oriented cycle must still be broken
in order to obtain an admissible ideal.  One possible choice is
\[
  I'=\langle \eta\alpha^*\rangle,
\]
which kills the path
\[
  2\xrightarrow{\alpha^*}1\xrightarrow{\eta}3.
\]
Then \(\K Q'/I'\) is gentle.
\end{example}

\begin{example}[A source mutation]
\label{ex:source-mutation}
Let
\[
  Q:
  \qquad
  1\xrightarrow{a}2,
  \qquad
  1\xrightarrow{b}3,
  \qquad
  2\xrightarrow{c}4,
  \qquad
  I=0.
\]
This is a gentle algebra.  Mutate at the source \(k=1\).  No composite arrows
are created.  The arrows leaving \(1\) are reversed, so \(Q'\) has arrows
\[
  a^*:2\to 1,
  \qquad
  b^*:3\to 1,
  \qquad
  c:2\to 4.
\]
The mutated quiver is acyclic, and one may take \(I'=0\).  Hence
\(\K Q'/I'\) is gentle.
\end{example}

\begin{nonexample}[A source mutation that is not degree-gentle]
\label{nonex:source-not-degree-gentle}
Let
\[
  Q:
  \qquad
  1\to 2,
  \qquad
  2\to 3,
  \qquad
  2\to 4.
\]
This quiver is degree-gentle.  Mutating at the source \(1\) only reverses the
arrow \(1\to 2\).  The mutated quiver has arrows
\[
  2\to 1,
  \qquad
  2\to 3,
  \qquad
  2\to 4.
\]
Now vertex \(2\) has three outgoing arrows.  Hence \(\mu_1(Q)\) is not
degree-gentle, and no choice of ideal on \(\K\mu_1(Q)\) can make the resulting
algebra gentle.
\end{nonexample}

\begin{nonexample}[Degree-gentleness can fail after adding composite arrows]
\label{nonex:composite-degree-failure}
Let
\[
  Q:
  \qquad
  1\to 2,
  \qquad
  2\to 3,
  \qquad
  2\to 4,
  \qquad
  1\to 5.
\]
This quiver is degree-gentle.  Mutating at \(2\) creates arrows \(1\to 3\) and
\(1\to 4\), reverses the arrows incident with \(2\), and leaves \(1\to 5\)
unchanged.  Thus vertex \(1\) has outgoing arrows to \(3\), \(4\), and \(5\),
so
\[
  \deg^+_{\mu_2(Q)}(1)=3.
\]
Again the mutated quiver cannot underlie a gentle algebra.
\end{nonexample}

\begin{example}[A mutation relation that is not an isomorphism]
The algebras in Example~\ref{ex:pathzero} are gentle-mutation-related, but
they are not isomorphic.  Indeed, the original quiver is a directed path with
two arrows, while the mutated quiver is an oriented 3-cycle.  Thus the
construction ~\ref{constr:gentle-admissible-completion} produces a new gentle algebra in the same quiver-mutation class,
but not the same algebra.
\end{example}

\begin{example}[A derived-equivalent source mutation]
The source mutation in Example~\ref{ex:source-mutation} is hereditary on the
part of the quiver affected by the mutation.  In such source or sink cases,
classical reflection functors often produce derived equivalences.  This shows
that the relation produced by gentle mutation may sometimes be stronger than a
purely combinatorial mutation relation, although such a derived equivalence is
not automatic in general.
\end{example}

\section{Questions and Conclusion}

\subsection{Questions for further study}

The preceding discussion motivates the following questions.

\begin{question}[Mutation-closed gentle classes]
Classify the gentle bound quivers \((Q,I)\) for which every mutation sequence
that preserves degree-gentleness can be accompanied by choices of relations so
that every resulting bound quiver remains gentle.
\end{question}

\begin{question}[Relation-sensitive rooted mutation groups]
For a gentle bound quiver \((Q,I)\), determine the subgroup
\[
  M_{\mathrm{rel}}(Q,I)\subseteq M(Q).
\]
When is this subgroup equal to \(M(Q)\), and when is it proper?
\end{question}

\begin{question}[Fiber mutation groups]
Let \(R=Q\otimes Q'\).  How do the rooted mutation groups of the row and
column subcluster structures compare with \(M(Q')\) and \(M(Q)\)?  More
precisely, when do mutation loops supported on a single row or column extend
to relation-preserving rooted loops of the full product quiver?
\end{question}

\begin{question}[Product-preserving mutation]
Classify mutation sequences \(\mu\) such that
\[
  \mu(Q\otimes Q')
\]
is again recognizable as an internal product.  How does this condition interact
with the existence of gentle ideals on the intermediate mutated quivers?
\end{question}

\begin{question}[Valued quivers and species]
For valued quivers, the ordinary path algebra should be replaced by a species
or a modulated quiver algebra.  Is there a valued or species-theoretic analogue
of the local matching and admissibility criteria developed in this paper?
\end{question}

\subsection{Conclusion}

Quiver mutation alone does not determine a mutation of gentle algebras.  The
reason is that a gentle algebra is not only a quiver, but a bound quiver
algebra
\[
  A=\K Q/I,
\]
where \(I\) is an admissible ideal generated by quadratic monomial
zero-relations.  Mutating \(Q\) produces a new quiver \(\mu_k(Q)\), but it does
not canonically specify the new ideal.

The results of this paper separate the problem into two parts.  The first part
is purely quiver-theoretic.  The mutated quiver must remain degree-gentle,
meaning that every vertex has at most two incoming and at most two outgoing
arrows.  Proposition~\ref{prop:degree-criterion} gives a local criterion for
this condition in the finite simple \(2\)-acyclic case.  The second part is
relation-theoretic.  Once \(\mu_k(Q)\) is degree-gentle, one must choose
quadratic zero-relations so that the local gentle matching conditions hold and
so that the resulting ideal is admissible.  Proposition~\ref{prop:matching}
provides the local test, while Proposition ~\ref{prop:admissibility} gives the
global admissibility test using the permitted transition graph.

The constructions introduced in this paper give practical ways to produce such
ideals.  Construction~\ref{constr:admissibilityfirst} prioritizes
admissibility, while Construction  ~\ref{constr:gentle-admissible-completion}
prioritizes completing the local gentle conditions.  Together, they show that
although mutation of gentle algebras is not automatic, it can often be carried
out by combining local degree control, local relation matching, and global
cycle-breaking conditions.

The final sections also clarify what kind of relation should be expected
between the original algebra \(\K Q/I\) and a mutated gentle algebra
\[
  \K\mu_k(Q)/I'.
\]
In general, the two algebras need not be isomorphic, and they need not be
derived equivalent.  Isomorphism occurs only in special cases where the bound
quivers agree up to isomorphism.  Derived equivalence may occur when the
constructed relation mutation agrees with a tilting or algebra mutation, but it
is not guaranteed by the construction itself.

This point is important for the connection with cluster theory.  At the level
of quivers, \(Q\) and \(\mu_k(Q)\) lie in the same mutation class.  At the
level of gentle algebras, however, the admissible ideal carries additional
information.  This suggests refining rooted mutation groups by introducing
relation-sensitive rooted mutation groups \(M_{\mathrm{rel}}(Q,I)\), consisting
of mutation loops that preserve the bound quiver structure and not merely the
underlying cluster quiver.

Overall, the paper provides a framework for studying mutation classes of
gentle bound quivers.  The main contribution is not a claim that gentleness is
mutation-invariant, but rather a set of precise local and global tests
describing when gentleness can be restored after mutation by a compatible
choice of relations.  These tests open the way to further study of
mutation-closed gentle classes, relation-sensitive rooted mutation groups,
and tensor-product constructions for gentle and valued quiver algebras.

\subsection*{Acknowledgments} I would like to thank Eman Alluqmani for introducing me to the topic of gentle algebras.

\end{document}